\documentclass[12 pt]{amsart}

\usepackage{booktabs}
\usepackage[margin=2.3cm]{geometry} 
\usepackage[T1]{fontenc}
\usepackage[utf8]{inputenc}
\usepackage{lmodern}
\usepackage{microtype}
\usepackage{amsmath,amssymb,amsthm,mathtools}
\usepackage{enumitem}
\usepackage[colorlinks=true,
linkcolor=blue,
citecolor=blue,
urlcolor=blue]{hyperref}

\newtheorem{theorem}{Theorem}[section]

\theoremstyle{definition}
\newtheorem{definition}[theorem]{Definition}

\theoremstyle{remark}

\newcommand{\R}{\mathbb{R}}

\newcommand{\rhoT}{\rho_T}

\begin{document}

	\title[$P$-Contractions that are not Enriched  ]
	{On The Existence of \(P\)-Contractions  that are not
		Enriched Contractions}
	
	\author{Faruk Temur}
	\address{Department of Mathematics\\
		Izmir Institute of Technology, Urla, Izmir, 35430,
		Turkey}
	\email{faruktemur@iyte.edu.tr}
	\date{August 20, 2026}

	\maketitle
	
	\begin{abstract}
		We construct two  examples clarifying the position of enriched
		$P$-contractions among related classes of mappings on normed spaces. Our first example is a  $P$-contraction operator on a finite subset of a normed linear space that is not an enriched contraction. In the second example we have a continuous operator on the whole of a normed linear space that is a $P$-contraction but not an enriched contraction. These examples answer in negative a question raised in \cite{AltunHancerAtes2024} asking whether every enriched $P$-contraction is an enriched contraction.
		
	\end{abstract}
	
	\section{Introduction}
	
	The Banach contraction principle is one of the fundamental results of fixed
	point theory.  Many subsequent developments replace the classical
	Lipschitz condition in the definition of a contraction by inequalities involving additional quantities associated
	with the mapping.
	Popescu introduced the class of $P$-contractions, in which the usual distance
	between two points is supplemented by the difference of their displacements
	under the mapping \cite{Popescu2008}. For subsets of linear normed spaces the precise definition  is as follows.
	\begin{definition}
		Let $C$ be a nonempty subset of a normed linear space
	$(X,\lVert\cdot\rVert)$, and let $T:C\to C$ be a mapping.	We call $T$  a $P$-contraction if there exists
		$\beta\in[0,1)$ such that,	for all $x,y\in C$
		\begin{equation}
			\lVert Tx-Ty\rVert
			\leq
			\beta\left(
			\lVert x-y\rVert+
			\left|\lVert x-Tx\rVert-\lVert y-Ty\rVert\right|
			\right).
			\label{eq:pcontraction}
		\end{equation}
			\end{definition}
	Berinde and P\u{a}curar 	\cite{BerindePacurar2020}  introduced
	enriched contractions on normed spaces; the enrichment parameter permits one
	to study mappings whose suitable averaged versions are contractions.
	\begin{definition}
		Let $C$ be a nonempty subset of a normed linear space
	$(X,\lVert\cdot\rVert)$, and let $T:C\to C$ be a mapping.	We call $T$  an enriched contraction if there exist constants
		$b\geq0$ and $\theta\in[0,b+1)$ such that, for all $x,y\in C$,
		
		\begin{equation}
			\lVert b(x-y)+Tx-Ty\rVert
			\leq
			\theta\lVert x-y\rVert.
			\label{eq:enriched-contraction}
		\end{equation}
		\end{definition}
	Altun, Aslan Han\c{c}er, and Do\u{g}an Ate\c{s}
	combined these two ideas in the notion of an enriched $P$-contraction
	\cite{AltunHancerAtes2024}.
	
	\begin{definition}
		Let $C$ be a nonempty subset of a normed linear space
	$(X,\lVert\cdot\rVert)$, and let $T:C\to C$ be a mapping.	We call $T$ an enriched $P$-contraction if there exist constants
		$b\geq0$ and $\theta\in[0,b+1)$ such that, 	for all $x,y\in C$
		
		\begin{equation}
			\begin{aligned}
				\lVert b(x-y)+Tx-Ty\rVert
				\leq
				\theta\lVert x-y\rVert 
				+
				\frac{\theta}{b+1}
				\left|\lVert x-Tx\rVert-\lVert y-Ty\rVert\right|.
			\end{aligned}
			\label{eq:enriched-p-contraction}
		\end{equation}

	\end{definition}

	The purpose of this note is to answer  a question raised in \cite{AltunHancerAtes2024}, and  show that 
there are enriched $P$-contractions that are not enriched contractions.  We note that for linear operators defined on all of a normed linear space, testing \eqref{eq:pcontraction},\eqref{eq:enriched-p-contraction} with $y=-x$ shows  that a $P$-contraction is a contraction, and an enriched  $P$-contraction is an enriched contraction. Therefore any  enriched $P$-contraction that is not an enriched contraction cannot be a linear operator on entirety of a normed linear space. 

We will give two examples, both of which live in the normed linear space $(\R^2,\|\cdot\|_{\infty}).$ The first  is a $P$-contraction defined on a 5 element subset of this space that is not an enriched contraction. Our  second example is more striking. It is a continuous $P$-contraction defined on all of this space, and yet it still is not an enriched contraction. Combining with the results of \cite{AltunHancerAtes2024}, these examples completely clarify  the implication relations of contractions (C), enriched contractions (EC), $P$-contractions (PC), enriched $P$-contractions (EPC). We summarize this     by 
\begin{equation*}
C\Rightarrow EC,PC \Rightarrow EPC    \qquad 	\qquad EC\nRightarrow PC,  \qquad  \qquad   PC\nRightarrow EC.
\end{equation*}
The large quantity of articles citing \cite{BerindePacurar2020,Popescu2008} testifies to the importance of the these newer versions of contractions. 
 Our article thus clarifies the relation between these intensely studied objects of functional analysis.

	\section{The discrete counterexample}
	
	This example is on a finite subset of $(\R^2,\|\cdot\|_{\infty})$, so verification of \eqref{eq:pcontraction} reduces to a finite computation.

	\begin{theorem}\label{thm:counterexample}
		Consider the  normed space $(\R^2,\lVert\cdot\rVert_\infty)$, and the subset $C=\{r,p,q,u,v\}$ where
		\begin{equation*}
				r=\left(1/2,0\right),\qquad
			p=(1,0),\qquad
			q=(0,0),\qquad
			u=(1,2),\qquad
			v=(0,1).
		\end{equation*}
			Define $T\colon C\to C$ by
		\begin{equation}\label{eq:defT}
			T r=r,\qquad
			T p=r,\qquad
			T q=r,\qquad
			T u=p,\qquad
			T v=q.
		\end{equation}
Then \(T\)
		is a \(P\)-contraction with constant $\beta=1/2$ , but  not an enriched contraction.
	\end{theorem}
	
	\begin{proof}
		The   displacement of $T$ is defined by
		$$
		\rhoT(x):=\lVert x-Tx\rVert_\infty.
		$$
		A direct calculation gives
		\begin{equation*}\label{eq:displacements}
			\rhoT(r)=0,\qquad
			\rhoT(p)=\rhoT(q)=1/2,\qquad
			\rhoT(u)=2,\qquad
			\rhoT(v)=1.
		\end{equation*}
		We show that $T$ is a $P$-contraction with constant $1/2$ by verifying the inequality \eqref{eq:pcontraction} for distinct points $x,y$. We remark that  for $x=y$ it is immediate, and by symmetry it suffices to check for only one of   the pairs $(x,y),(y,x).$ 
			\begin{table}[ht]
			\centering
			\small
			\caption{Verification of the enriched \(P\)-contractive inequality.}
			\label{tab:pairs}
			\begin{tabular}{ccc}
				\toprule
				\(\{x,y\}\)
				& \(\lVert Tx-Ty\rVert_\infty\)
				& \(\frac12\bigl(\lVert x-y\rVert_\infty+
				|\rhoT(x)-\rhoT(y)|\bigr)\) \\
				\midrule
				\(\{r,p\}\) & \(0\)          & \(1/2\) \\
				\(\{r,q\}\) & \(0\)          & \(1/2\) \\
				\(\{p,q\}\) & \(0\)          & \(1/2\) \\
				\(\{r,v\}\) & \(1/2\)    & \(1\) \\
				\(\{r,u\}\) & \(1/2\)    & \(2\) \\
				\(\{p,v\}\) & \(1/2\)    & \(3/4\) \\
				\(\{q,v\}\) & \(1/2\)    & \(3/4\) \\
				\(\{p,u\}\) & \(1/2\)    & \(7/4\) \\
				\(\{q,u\}\) & \(1/2\)    & \(7/4\) \\
				\(\{u,v\}\) & \(1\)          & \(1\) \\
				\bottomrule
			\end{tabular}
		\end{table}

		 Assume to the contrary $T$ is an enriched contraction, and therefore satisfies \eqref{eq:enriched-contraction} with some constants  $b\geq 0,\  \theta\in[0,b+1)$.
		 As  we have from 	\eqref{eq:defT}
			\[
		u-v=(1,1),
		\qquad
		Tu-Tv=p-q=(1,0),
		\qquad
		\lVert u-v\rVert_\infty=1,
		\]
		 applying \eqref{eq:enriched-contraction} to $u,v$ gives the following contradiction
		\begin{equation*}
			\lVert b(u-v)+Tu-Tv\rVert_\infty
			=\lVert(b+1,b)\rVert_\infty
			=b+1\leq \theta.
		\end{equation*}
		So $T$ cannot be an enriched contraction.
		
	\end{proof}

	\section{A nonlinear example on the whole space}
	
	With the next construction we aim to demonstrate the existence of $P$-contractions satisfying additional properties that still are not enriched contractions.  First of all, we want an example defined on all of the ambient normed space, as examples defined on arbitrarily contrived   subsets  look  unnatural.  
	Secondly, as a linear example on entirety of a  normed linear space cannot exist, continuity is the next best property we may expect.

	\begin{theorem}
		Consider the  normed space $(\R^2,\lVert\cdot\rVert_\infty)$. We define $T:\R^2\to \R^2$ using a function
		 $h:\R^2\to[0,\infty)$ as follows:
		\begin{equation}
			T(s,t)=\left(\frac{s+h(s,t)}{2},0\right), \qquad     \qquad 	h(s,t)=\max\left\{|t|,s/3,-s\right\}.
			\label{eq:def-nonlinear-T}
		\end{equation}
			Then $T$ is continuous and is a $P$-contraction with constant
		$\beta=1/2$. Nevertheless, $T$ is not an enriched contraction.
		\label{thm:whole-space-example}
	\end{theorem}
	
	\begin{proof}
		The function $h$ is the maximum of three continuous real-valued functions, so
		$h$ is continuous. Hence $T$ is continuous. We first show that  its displacement
		function $\rho_T$ is actually equal to $h$. From the definition of $h$
we have		
		\begin{equation*}
			-h(s,t)\leq s\leq 3h(s,t)  \qquad   \implies  	\qquad  \frac{|s-h(s,t)|}{2}\leq h(s,t).
		\end{equation*}
	As  by definition of $h$ 	 we already know  $|t|\leq h(s,t)$, we have  
		\begin{equation}
			\rhoT(s,t)=\|(s,t)-T(s,t)\|_{\infty} =\Big\|	\left(\frac{s-h(s,t)}{2},t\right)\Big\|_{\infty}
			=
			\max\left\{\frac{|s-h(s,t)|}{2},|t|\right\}
			\leq h(s,t).
			\label{eq:rho-upper}
		\end{equation}
		
		At least one of the three quantities in the maximum defining $h$ is equal to
		$h$. If $h(s,t)=|t|$, then 
			\begin{equation*}
			\rhoT(s,t)=
			\max\left\{\frac{|s-h(s,t)|}{2},|t|\right\}
			\geq   |t|=h(s,t),
				\end{equation*}
		yielding, when combined with \eqref{eq:rho-upper} the equality of $h$ and $\rho_T$.
		
		 If $h(s,t)=s/3$, then $s=3h(s,t)$ and plugging this  we get
		\begin{equation*}
		\begin{aligned}
		\rhoT(s,t)=
		\max\left\{\frac{|s-h(s,t)|}{2},|t|\right\}=
		\max\left\{|h(s,t)|,|t|\right\}
		\geq   h(s,t),
	\end{aligned}
	\end{equation*}
		again giving equality.

	Finally 	if $h(s,t)=-s$, then $s=-h(s,t)$, and  again
		\begin{equation*}
			\begin{aligned}
				\rhoT(s,t)=
				\max\left\{\frac{|s-h(s,t)|}{2},|t|\right\}=
				\max\left\{|h(s,t)|,|t|\right\}
				\geq   h(s,t),
			\end{aligned}
		\end{equation*}
	giving equality.
	
		Thus we have established
		$
			\rhoT(s,t)
			=
			h(s,t).
			\label{eq:rho-equals-h}
		$
	on all of $\R^2.$	
		Now pick two arbitrary points $x=(x_1,x_2)$ and $y=(y_1,y_2)$. By definition of $T$ and the triangle inequality

	\begin{equation*}
			\begin{aligned}
				\lVert Tx-Ty\rVert_\infty =	\Big\| \left( \frac{x_1-y_1+h(x_1,x_2)-h(y_1,y_2)}{2},0    \right) \Big\|_\infty  &=\left| \frac{x_1-y_1+h(x_1,x_2)-h(y_1,y_2)}{2}   \right| \\ &\leq \frac{1}{2}\|x-y\|_{\infty}+\frac{1}{2}|\rho_T(x)-\rho_T(y)|.
			\end{aligned}
		\end{equation*}
		By definition of the displacement, this shows that  $T$ is a $P$-contraction with constant $1/2$. 
		
		It remains to show that $T$ is not an enriched contraction. Consider the
		points	$x=(1,2),$
			and
			$y=(0,1).$
		Their $h$-values are
		$h(1,2)=2,$
			and 
			$h(0,1)=1$.
		Hence $
			Tx=\left(3/2,0\right),$
			and
			$Ty=\left(1/2,0\right). $
			It follows that
			\begin{equation*}
			x-y=(1,1),
			\qquad
			Tx-Ty=(1,0),
			\qquad
			\lVert x-y\rVert_\infty=1.
		\end{equation*}
		
	Applying the inequality \eqref{eq:enriched-contraction} defining enriched contractions to $T$ with these points $x,y$ yields for any  $b\geq 0$ 
			\begin{equation*}
			\begin{aligned}
				\lVert b(x-y)+Tx-Ty\rVert_\infty
				=
				\lVert(b+1,b)\rVert_\infty 
				=b+1=(b+1)\lVert x-y\rVert_\infty.
			\end{aligned}
		\end{equation*}
Hence $T$ cannot satisfy  \eqref{eq:enriched-contraction}  for any $\theta\in [0,b+1)$, and therefore is not an enriched contraction.
	\end{proof}

\end{document}